\documentclass[11pt]{article}
\usepackage{graphicx}
\usepackage{amsmath,amsthm,amssymb,amsfonts,euscript,enumerate}
\usepackage{amsthm}
\usepackage{color,wrapfig}
\usepackage{pxfonts}
\usepackage{verbatim}
\newtheorem{thm}{Theorem}[section]
\newtheorem{cor}[thm]{Corollary}
\newtheorem{lem}[thm]{Lemma}
\newtheorem{prop}[thm]{Proposition}

\newcommand{\3}{\varepsilon}
\newcommand{\4}{\widetilde}

\newcommand{\Rr}{{\mathbb{R}}}

\newcommand{\Z}{{\mathbb Z}}

\newcommand{\cD}{{\mathcal D}}

\newcommand{\cX}{{\mathcal X}}

\newcommand{\ve}{\varepsilon}

\begin{document}
\title{Existence of singular rotationally symmetric gradient Ricci solitons in higher dimensions}
\author{Kin Ming Hui\\
%\thanks{ }\\
Institute of Mathematics, Academia Sinica\\
Taipei, Taiwan, R. O. C.}
\date{March 15, 2024}
\smallbreak \maketitle
\begin{abstract}
By using fixed point argument we give a proof for the existence of singular rotationally symmetric steady and expanding gradient Ricci solitons in higher dimensions with metric $g=\frac{da^2}{h(a^2)}+a^2g_{S^n}$ for some function $h$ where $g_{S^n}$ is the standard metric on the unit sphere $S^n$ in $\mathbb{R}^n$ for any $n\ge 2$. More precisely for any $\lambda\ge 0$ and $c_0>0$, we prove that there exist infinitely many solutions $h\in C^2((0,\infty);\mathbb{R}^+)$ for the equation $2r^2h(r)h_{rr}(r)=(n-1)h(r)(h(r)-1)+rh_r(r)(rh_r(r)-\lambda r-(n-1))$, $h(r)>0$, in $(0,\infty)$ satisfying $\underset{\substack{r\to 0}}{\lim}\,r^{\sqrt{n}-1}h(r)=c_0$ and prove the higher order asymptotic behaviour of the global singular solutions   near the origin. We also find conditions for the existence of unique global singular solution of such equation in terms of its asymptotic behaviour near the origin. 
\end{abstract}

\vskip 0.2truein



\vskip 0.2truein
\setcounter{equation}{0}
\setcounter{section}{0}

\section{Introduction}
\setcounter{equation}{0}
\setcounter{thm}{0}

Ricci flow is an important technique in geometry and has a lot of applications in geometry \cite{KL}, \cite{MT}, \cite{P1}, \cite{P2}. For example recently G.~Perelman  \cite{P1}, \cite{P2}, used Ricci flow to prove the Poincare conjecture. In the study of Ricci flow one is interested to study the Ricci solitons which are self-similar solutions of Ricci flow.
On the other hand by a limiting argument the behaviour of the Ricci flow near the singular time is 
usually similar to the behaviour of Ricci solitons. 

Hence in order to understand Ricci flow it is important to study the Ricci solitons. 
In \cite{B2} S.~Brendle used singular rotationally symmetric steady solitons to construct barrier functions which plays an important role in the proof there that confirms a conjecture of Perelman on $3$-dimensional ancient $\kappa$ solution to the Ricci flow. 
We refer the reader to the papers by  S.~Alexakis, D.~Chen and G.~Fournodavlos, S.~Brendle,  R.L.~Bryant,  H.D.~Cao,   D.~Zhou,  M.~Feldman, T.~Ilmanen and D.~Knopf, S.Y.~Hsu, Y.~Li and B.~Wang,  O.~Munteanu and N.~Sesum, P.~Petersen and W.~Wylie, \cite{ACF}, \cite{B1}, \cite{Br}, \cite{C}, \cite{CZ}, \cite{FIK}, \cite{H}, \cite{LW}, \cite{MS}, \cite{PW}, etc. 
and the book [CCGG] by B.~Chow, S.C.~Chu, D.~Glickenstein, C.~Guenther, J.~Isenberg, T.~Ivey, D.~Knopf, P.~Lu, F.~Luo and L.~Ni for some recent results on Ricci solitons.

We say that a Riemannian metric $g=(g_{ij})$ on a Riemannian manifold $M$  is a gradient Ricci soliton if there exist a smooth function $f$ on $M$ and a constant $\lambda\in\mathbb{R}$ such that the Ricci curvature $R_{ij}$ of the metric $g$ satisfies
\begin{equation}\label{gradient-soliton-eqn}
R_{ij}=\nabla_i\nabla_jf-\lambda g_{ij}\quad\mbox{ on }M.
\end{equation}
The gradient soliton is called an expanding gradient Ricci soliton if $\lambda>0$. It is called a steady gradient  Ricci soliton if $\lambda=0$ and it is called a shrinking gradient  Ricci soliton if $\lambda<0$.

Existence of rotationally symmetric steady and expanding $3$-dimensional gradient Ricci solitons were proved by R.L.~Bryant \cite{Br} using the phase method and by S.Y.~Hsu \cite{H} using fixed point argument. On the other hand as observed by R.L.~Bryant \cite{Br} for $n=2$ and B.~Chow, etc. (cf. Lemma 1.21 and Section 4 of Chapter 1 of \cite{CCGG}) for $n\ge 2$, for any $n\ge 2$ if $(M,g)$ is a $(n+1)$-dimensional rotational symmetric gradient Ricci soliton which satisfies \eqref{gradient-soliton-eqn} for some smooth function $f$ and constant $\lambda\in\mathbb{R}$
with
\begin{equation}\label{g-rotational-form}
g=dt^2+a(t)^2\,g_{S^n}
\end{equation}
where $g_{S^n}$ is the standard metric on the unit sphere $S^n$ in $\mathbb{R}^n$, then the Ricci curvature of $g$ is given by
\begin{equation}\label{ric-=}
Ric(g)=-\frac{na_{tt}(t)}{a(t)}\,dt^2+\left(n-1-a(t)a_{tt}(t)-(n-1)a_t(t)^2\right)\,g_{S^n}
\end{equation}  
and
\begin{equation}\label{hess-f}
Hess(f)=f_{tt}(t)\,dt^2+a(t)a_t(t)f_t(t)\,g_{S^n}.
\end{equation}
Hence by \eqref{gradient-soliton-eqn}, \eqref{ric-=} and \eqref{hess-f} (cf. \cite{ACF}, \cite{Br}, \cite{CCGG}), we get
\begin{equation}\label{a-f-eq1}
-na(t)a_{tt}(t)=a(t)(f_{tt}(t)-\lambda)
\end{equation}
and
\begin{equation}\label{a-f-eq2}
n-1-a(t)a_{tt}(t)-(n-1)a_t(t)^2=a(t)a_t(t)f_t(t)-\lambda a(t)^2.
\end{equation}
By eliminating  $f$ from \eqref{a-f-eq1} and \eqref{a-f-eq2} we get that $a(t)$ satisfies
\begin{equation}\label{a-eqn}
a(t)^2a_t(t)a_{ttt}(t)=a(t)a_t(t)^2a_{tt}(t)+a(t)^2a_{tt}(t)^2-(n-1)a(t)a_{tt}(t)-\lambda a(t)^3a_{tt}(t)-(n-1)a_t(t)^2+(n-1)a_t(t)^4.
\end{equation}
Note that we can express $g$ as
\begin{equation}\label{g=h-form}
g=\frac{da^2}{h(a^2)}+a^2\,g_{S^n}
\end{equation}
where $h(r)$, $r=a^2\ge 0$, and $a=a(t)$ satisfies
\begin{equation}\label{a-t-relation}
a_t(t)=\sqrt{h(a(t)^2)}.
\end{equation}
Then by \eqref{a-eqn} and a direct computation $h$ satisfies
\begin{equation}\label{h-eqn}
2r^2h(r)h_{rr}(r)=(n-1)h(r)(h(r)-1)+rh_r(r)(rh_r(r)-\lambda r-(n-1)),\quad h(r)>0.
\end{equation}
We are now interested in rotational symmetric gradient Ricci soliton which blows up at $r=0$
at the rate
\begin{equation}\label{h-origin-blow-up-rate}
\lim_{r\to 0}r^{\alpha}h(r)=c_0
\end{equation}
for some constants $\alpha>0$ and $c_0>0$. Let 
\begin{equation}\label{w-defn}
w(r)=r^{\alpha}h(r)\quad\forall r>0.
\end{equation}
By \eqref{h-eqn}, \eqref{w-defn} and a direct computation $w$ satisfies
\begin{align}\label{w-eqn1}
2r^2w(r)w_{rr}(r)=&2\alpha r w(r)w_r(r)+(n-1)(\alpha -1)r^{\alpha}w(r)+\alpha\lambda r^{\alpha +1}w(r)-(n-1)r^{\alpha +1}w_r(r)\notag\\
&\quad -\lambda r^{\alpha +2}w_r(r)+r^2w_r(r)^2-(\alpha^2+2\alpha -(n-1))w(r)^2.
\end{align}
Unless stated otherwise we now let $\alpha=\sqrt{n}-1>0$ for the rest of the paper. Then $\alpha^2+2\alpha -(n-1)=0$.  Hence by \eqref{w-eqn1} $w$ satisfies
\begin{align}\label{w-eqn2}
2r^2w(r)w_{rr}(r)=&2\alpha r w(r)w_r(r)+(n-1)(\alpha -1)r^{\alpha}w(r)+\alpha\lambda r^{\alpha +1}w(r)-(n-1)r^{\alpha +1}w_r(r)\notag\\
&\quad -\lambda r^{\alpha +2}w_r(r)+r^2w_r(r)^2
\end{align}
with $\alpha=\sqrt{n}-1>0$. We also impose the condition
\begin{equation}\label{a0=0}
\lim_{t\to 0^+}a(t)=0.
\end{equation}
Then by \eqref{a-t-relation} and \eqref{a0=0},
\begin{equation}\label{t-at-relation}
t=\int_0^{a(t)}\frac{d\rho}{\sqrt{h(\rho^2)}}.
\end{equation}
In the paper \cite{Br} R.L.~Bryant by using power series expansion around the singular point at the origin gave the local existence of singular solution of \eqref{h-eqn} near the origin which blows up at the rate \eqref{h-origin-blow-up-rate} for the case $n=2$. On the other hand by using phase plane analysis of the functions
\begin{equation*}
W=\frac{1}{f_t(t)+n\frac{a_t(t)}{a(t)}},\quad X=\sqrt{n}W\frac{a_t(t)}{a(t)},\quad Y=\frac{\sqrt(n-1)W}{a(t)},
\end{equation*}
S.~Alexakis, D.~Chen and G.~Fournodavlos, [ACF] gave a sketch of proof for the local existence of singular solution $(a(t), f(t))$, of \eqref{a-f-eq1},\eqref{a-f-eq2}, near the origin and its the asymptotic behaviour  as $t\to 0^+$ for the case $n\ge 2$. When $\lambda=0$, existence of global solution $(a(t), f(t))$, of \eqref{a-f-eq1},\eqref{a-f-eq2}, in $(0,\infty)$ is also mentioned without detailed proof in \cite{ACF}.

In this paper we will use fixed point argument for the function $w$ given by \eqref{w-defn} to give a new proof of the local existence of solution  $h$ of \eqref{h-eqn} satisfying \eqref{h-origin-blow-up-rate} for any constants $\lambda\in\Rr$, $c_0>0$, and $2\le n\in\Z^+$. For $\lambda\ge 0$ we will then use continuation method to extend the local singular solutions of \eqref{h-eqn}, \eqref{h-origin-blow-up-rate}, to global solutions of \eqref{h-eqn}, \eqref{h-origin-blow-up-rate}. We will also prove the higher order asymptotic behaviour of the local solutions  of \eqref{h-eqn}, \eqref{h-origin-blow-up-rate}, near the origin.

The main results we obtain in this paper are the following.

\begin{thm}\label{h-existence-thm}
Let $2\le n\in\Z^+$, $\lambda\ge 0$, $\alpha=\sqrt{n}-1$, $c_0>0$,  $c_1\in\Rr$ and
\begin{equation}\label{c2-defn}
c_2:=\frac{(n-1)(\alpha-1)}{2}=\frac{(n-1)(\sqrt{n}-2)}{2}.
\end{equation} 
There exists a  unique solution $h\in C^2((0,\infty))$ of \eqref{h-eqn} in $(0,\infty)$ which satisfies \eqref{h-origin-blow-up-rate} and \eqref{w-eqn=w-integral-eqn} in $(0,\ve)$ with $w$ given by \eqref{w-defn} for some constant $\ve>0$. 
\end{thm} 

\begin{thm}\label{h-asymptotic-behaviour-thm1}
Let $4<n\in\Z^+$,  $\lambda\ge 0$, $\alpha=\sqrt{n}-1$, $c_0>0$,  $c_1\in\Rr$ and
$c_2$ be given by \eqref{c2-defn}.
Then there exists a constant $0<\delta_0<1$ such that \eqref{h-eqn}  has a unique solution $h\in C^2((0,\infty))$  in $(0,\infty)$ which satisfies \eqref{h-origin-blow-up-rate} and 
\begin{align}\label{h-asymptotic-near-origin1}
h(r)=&\frac{1}{r^{\alpha}}\left\{c_0-\frac{c_2}{\alpha}r^{\alpha}+\left(\frac{c_1}{\alpha+1}-\frac{\alpha\lambda}{2(\alpha +1)^2}\right)r^{\alpha +1}+\frac{\alpha\lambda}{2\alpha +2}r^{\alpha +1}\log r+\frac{c_2^2+(n-1)c_2}{4c_0\alpha (\alpha -1)}r^{2\alpha}\right.\notag\\
&\quad \left.+o(1)r^{2\alpha}\right\}\quad\forall 0<r\le \delta_0. 
\end{align}
Moreover
\begin{align}\label{hr-asymptotic-near-origin1}
h_r(r)=&\frac{1}{r^{\alpha+1}}\left\{-\alpha c_0+\left(\frac{c_1}{\alpha+1}+\frac{\alpha^2\lambda}{2(\alpha +1)^2}\right)r^{\alpha +1}+\frac{\alpha\lambda}{2\alpha +2}r^{\alpha +1}\log r+\frac{c_2^2+(n-1)c_2}{4c_0(\alpha -1)}r^{2\alpha}\right.\notag\\
&\quad \left.+o(1)r^{2\alpha}\right\}\quad\forall 0<r\le \delta_0.
\end{align}
\end{thm} 

\begin{thm}\label{h-asymptotic-behaviour-thm2}
Let $n\in\{2,3,4\}$, $\alpha=\sqrt{n}-1$, $\lambda\ge 0$, $c_0>0$, $c_1\in\Rr$ and
$c_2$ be given by \eqref{c2-defn}.
Let $h\in C^2((0,\infty))$ be  given by Theorem \ref{h-existence-thm}. Then there exists a constant $0<\delta_0<1$ such that
\begin{equation}\label{h-asymptotic-near-origin2}
h(r)=\left\{\begin{aligned}
&\frac{1}{r^{\alpha}}\left(c_0-\frac{c_2}{\alpha}r^{\alpha}-\frac{c_2(c_2+n-1)}{4c_0\alpha (1-\alpha)}r^{2\alpha}+o(1)r^{2\alpha}\right)
\quad\forall 0<r\le \delta_0\quad\mbox{ if }n=2,3\\
&\frac{1}{r}\left(c_0+\frac{\lambda}{4}r^2\log r+o(1)r^2|\log r|\right)
\qquad\qquad\quad\forall 0<r\le \delta_0\quad\mbox{ if }n=4.
\end{aligned}\right.
\end{equation}
Moreover
\begin{equation}\label{hr-asymptotic-near-origin2}
h_r(r)=\left\{\begin{aligned}
&\frac{1}{r^{\alpha+1}}\left(-\alpha c_0-\frac{c_2(c_2+n-1)}{4c_0(1-\alpha)}r^{2\alpha}+o(1)r^{2\alpha}\right)\quad\forall 0<r\le \delta_0\quad\mbox{ if }n=2,3\\
&\frac{1}{r^2}\left(-c_0+\frac{\lambda}{4}r^2\log r+o(1)r^2|\log r|\right)\qquad\quad\forall 0<r\le \delta_0\quad\mbox{ if }n=4.
\end{aligned}\right.
\end{equation}

\end{thm} 

Note that  the singular solutions $h$ of \eqref{h-eqn} in $(0,\infty)$ given by Theorem \ref{h-existence-thm}, Theorem \ref{h-asymptotic-behaviour-thm1} and  Theorem \ref{h-asymptotic-behaviour-thm2} satisfies \eqref{h-origin-blow-up-rate} with $\alpha=\sqrt{n}-1$. Moreover by \eqref{g-rotational-form} the solitons constructed in Theorem \ref{h-existence-thm} and Theorem \ref{h-asymptotic-behaviour-thm1} are complete at $t=+\infty$.  A natural question to ask is that does there exist any other singular solution of \eqref{h-eqn} in $(0,\ve)$ for some constant $\ve>0$ which blow-up at a different rate at the origin. We answer this question in the negative. More precisely we prove the following result.

\begin{thm}\label{blow-up-rate-at-origin-thm}
Let $2\le n\in\Z^+$, $\lambda\in\Rr$, $\3>0$ and $c_0>0$. Suppose $h\in C^2((0,\3))$ is a solution of \eqref{h-eqn} in $(0,\3)$ which satisfies \eqref{h-origin-blow-up-rate} for some constant $\alpha>0$. Then $\alpha=\sqrt{n}-1$.
\end{thm}

The plan of the paper is as follows. In section two we will prove the local existence of infinitely many singular solutions of \eqref{h-eqn}, \eqref{h-origin-blow-up-rate}, in a neighbourhood of the origin  and conditions for uniqueness of local singular solutions are given. We will also prove the higher order asymptotic behaviour of these local solutions near the origin. In section three we will prove the global existence of infinitely many singular solutions of \eqref{h-eqn}, \eqref{h-origin-blow-up-rate} and conditions for uniqueness of global singular solution are given. In section four we will prove the asymptotic behaviour of $a(t)$ near the origin.

\section{Local existence, uniqueness and asymptotic behaviour of singular solutions near the origin}
\setcounter{equation}{0}
\setcounter{thm}{0}

In this section for any $2\le n\in\Z^+$, $\lambda\in\Rr$  and $c_0>0$, we will prove the local existence of infinitely many singular solutions of \eqref{h-eqn} in $(0,\ve)$ which satisfy \eqref{h-origin-blow-up-rate} for some constant $\ve>0$. Under some mild conditions on the singular solutions of \eqref{h-eqn} in $(0,\ve)$ we will also prove the uniqueness of local singular solutions of \eqref{h-eqn} in $(0,\ve)$ satisfying \eqref{h-origin-blow-up-rate}. We first observe that if $h\in C^2((0,\3];\Rr^+)$ is a solution of \eqref{h-eqn} in $(0,\3]$ for some constant $\3>0$ which satisfies \eqref{h-origin-blow-up-rate} for some constant $c_0>0$ and $w$ is given by \eqref{w-defn} with $\alpha=\sqrt{n}-1$
, then by \eqref{h-origin-blow-up-rate}, \eqref{w-defn} and \eqref{w-eqn2}, $w>0$  satisfies 
\begin{equation}\label{wrr-eqn}
w_{rr}(r)=\frac{\alpha}{r}w_r(r)+\frac{(n-1)(\alpha -1)}{2}r^{\alpha-2}+\frac{\alpha\lambda}{2} r^{\alpha-1}-\frac{(n-1)r^{\alpha-1}w_r(r)}{2w(r)}
-\frac{\lambda r^{\alpha}w_r(r)}{2w(r)}+\frac{w_r(r)^2}{2w(r)}
\end{equation}
in $(0,\3]$ and
\begin{equation}\label{w0=c0}
w(0)=c_0
\end{equation}
if $w\in C([0,\3];\Rr^+)$. Hence the existence of solution  $h\in C^2((0,\3];\Rr^+)$ of \eqref{h-eqn} in $(0,\3]$ which satisfies
\eqref{h-origin-blow-up-rate} is equivalent to the existence of solution $w\in C^2((0,\3];\Rr^+)\cap C([0,\3];\Rr^+)$ of \eqref{wrr-eqn} in $(0,\3]$ which satisfies \eqref{w0=c0}. Note that \eqref{wrr-eqn} is equivalent to 
\begin{align}
(r^{-\alpha}w_r)_r(r)=&c_2r^{-2}+\frac{\alpha\lambda}{2}r^{-1}
-\frac{(n-1)r^{-1} w_r(r)}{2w(r)}-\frac{\lambda w_r(r)}{2w(r)}+\frac{r^{-\alpha}w_r(r)^2}{2w(r)}\quad\forall 0<r\le\3\label{wrr-eqn2}\\
\Leftrightarrow\quad r^{-\alpha}w_r(r)=&-c_2r^{-1}+c_1+\frac{\alpha\lambda}{2}\log r
+\frac{(n-1)}{2}\int_r^{\ve}\frac{\rho^{-1}w_r(\rho)}{w(\rho)}\,d\rho-\frac{\lambda}{2}\int_0^r\frac{w_r(\rho)}{w(\rho)}\,d\rho\notag\\
&\quad -\frac{1}{2}
\int_r^{\ve}\frac{\rho^{-\alpha}w_r(\rho)^2}{w(\rho)}\,d\rho\qquad\qquad\forall 0<r\le\3\notag\\
\Leftrightarrow\qquad\,\, w_r(r)=&-c_2r^{\alpha-1}+c_1r^{\alpha}+\frac{\alpha\lambda}{2}r^{\alpha}\log r
+r^{\alpha}\left\{\frac{(n-1)}{2}\int_r^{\ve}\frac{\rho^{-1}w_r(\rho)}{w(\rho)}\,d\rho-\frac{\lambda}{2}\int_0^r\frac{w_r(\rho)}{w(\rho)}\,d\rho\right.\notag\\
&\quad \left.-\frac{1}{2}
\int_r^{\ve}\frac{\rho^{-\alpha}w_r(\rho)^2}{w(\rho)}\,d\rho\right\}\qquad\qquad\forall 0<r\le\3
\label{w-eqn=w-integral-eqn}
\end{align}
for some cosntant $c_1\in\Rr$.
This suggests one to use fixed point argument to prove the existence of solution $w\in C^2((0,\3];\Rr^+)\cap C([0,\3];\Rr^+)$ of \eqref{wrr-eqn} in $(0,\3]$ which satisfies \eqref{w0=c0}. 

\begin{prop}\label{w-local-existence-prop}
Let $2\le n\in\Z^+$, $\alpha=\sqrt{n}-1$, $\lambda, c_1\in\Rr$, $c_0>0$  and let $c_2$ be given by \eqref{c2-defn}.  Then there exists a constant $0<\3<1$ such that \eqref{wrr-eqn} has a unique solution $w\in C^2((0,\3];\Rr^+)\cap C([0,\3];\Rr^+)$ in $(0,\3]$ which satisfies   \eqref{w0=c0} and \eqref{w-eqn=w-integral-eqn}. Moreover
\begin{equation}\label{wr-initial-value}
\lim_{r\to 0^+}r^{1-\alpha}w_r(r)=-c_2
\end{equation}
holds.
\end{prop} 
\begin{proof}
For any $\3>0$ we define the Banach space 
\begin{align*}
\cX_\3:=&\{(w,v): w\in  C( [0,\3];\Rr^+),v\in C( (0,\3];\Rr)\mbox{ such that }r^{1-\alpha}v(r)\mbox{ can be extended to a}\\
&\mbox{ function in }C( [0,\3];\Rr)\}
\end{align*} 
with a norm given by
\begin{equation*}
\|(w,v)\|_{\cX_\3}=\max\left(\|w\|_{L^{\infty}\left([0, \3]\right)} ,\|r^{1-\alpha}v(r)\|_{L^{\infty}\left([0, \3]\right)} \right). 
\end{equation*}
For any $(w,v)\in \cX_\3,$ we define  
\begin{equation*}
\Phi(w,v):=\left(\Phi_1(w,v),\Phi_2(w,v)\right) 
\end{equation*}
where 
\begin{equation}\label{contraction-map-defn}
\left\{\begin{aligned}
\Phi_1(w,v)(r)=&c_0+\int_0^r v(\rho)\,d\rho,\\
\Phi_2(w,v)(r)=&-c_2r^{\alpha-1}+c_1r^{\alpha}+\frac{\alpha\lambda}{2}r^{\alpha}\log r
+r^{\alpha}\left\{\frac{(n-1)}{2}\int_r^{\ve}\frac{\rho^{-1}v(\rho)}{w(\rho)}\,d\rho
-\frac{\lambda}{2}\int_0^r\frac{v(\rho)}{w(\rho)}\,d\rho\right.\\
&\quad \left.-\frac{1}{2}
\int_r^{\ve}\frac{\rho^{-\alpha}v(\rho)^2}{w(\rho)}\,d\rho\right\}
\end{aligned}\right.
\end{equation}
for any $0<r\leq\3$. Let 
\begin{equation}\label{closed-set-defn}
\cD_{\3}:=\left\{\|(w,v)-(c_0,-c_2r^{\alpha-1})\|_{\cX_\3}\le c_0/10\right\}.
\end{equation}
Since $(c_0,-c_2r^{\alpha-1})\in\cD_{\3}$, $\cD_{\3}\ne\phi$.
We will show that  there exists    $\ve\in(0,1)$ such that the map  $(w,v)\mapsto\Phi(w,v)$ has a unique fixed point  in the closed subspace $\cD_{\3}$. Let
\begin{equation*}
\3_1=\min\left(\frac{1}{2}, \left(\frac{c_0\alpha}{10|c_2|+c_0}\right)^{1/\alpha}\right).
\end{equation*}
We first prove  that  $\Phi(\cD_{\3})\subset \cD_{\3}$ for sufficiently small $\ve\in (0,\ve_1)$. For any $\ve\in (0,\ve_1)$, $(w,v)\in \cD_{\3},$  $0\leq r<\ve $, by \eqref{closed-set-defn} we have
\begin{equation}\label{w-range}
\frac{9c_0}{10}\le w(r)\le \frac{11c_0}{10}\quad\forall 0<r\le\ve
\end{equation}
and
\begin{equation}\label{v-bd}
|v(r)|\le c_3r^{\alpha -1}\quad\forall 0<r\le\ve
\end{equation}
where $c_3=|c_2|+(c_0/10)$. Hence by \eqref{v-bd},
\begin{align}\label{phi1-bd}
&|\Phi_1(w,v)(r)-c_0|\le\int_0^rc_3\rho^{\alpha -1}\,d\rho=(c_3/\alpha)r^{\alpha}\le (c_3/\alpha)\ve^{\alpha}\le c_0/10\quad\forall 0<r\le\ve\notag\\
\Rightarrow\quad&\|\Phi_1(w,v)-c_0\|_{L^{\infty}\left([0, \3]\right)}\le c_0/10\quad\mbox{ and }\quad \|\Phi_1(w,v)\|_{L^{\infty}\left([0, \3]\right)}\le 11c_0/10.
\end{align}
We now choose $c_4>1$ such that
\begin{equation}\label{log-r-bd}
|\log r|\le c_4r^{-1/2}\quad\forall 0<r\le 1/2.
\end{equation}
Then by  \eqref{w-range}, \eqref{v-bd} and \eqref{log-r-bd} for any $0<r\le\ve$,
\begin{align}\label{v-w-integral-estimate1}
r\left|\int_r^{\ve}\frac{\rho^{-1}v(\rho)}{w(\rho)}\,d\rho\right|
\le&\frac{10c_3r}{9c_0}\int_r^{\ve}\rho^{\alpha-2}\,d\rho
\le\left\{\begin{aligned}
&\frac{10c_3r(r^{\alpha-1}+\ve^{\alpha-1})}{9c_0|\alpha-1|}\quad\,\mbox{ if }n \ne 4\\
&\frac{10c_3r|\log r|}{9c_0}\qquad\qquad\mbox{if }n=4
\end{aligned}\right.\notag\\
\le&\left\{\begin{aligned}
&\frac{20c_3\3^{\alpha}}{9c_0|\alpha -1|}\quad\,\mbox{ if }n \ne 4\\
&\frac{10c_3c_4\ve^{1/2}}{9c_0}\quad\mbox{ if }n=4,
\end{aligned}\right.
\end{align}
\begin{equation}\label{v-w-integral-estimate2}
r\left|\int_0^r\frac{v(\rho)}{w(\rho)}\,d\rho\right|
\le\frac{10c_3r}{9c_0}\int_0^r\rho^{\alpha-1}\,d\rho
=\frac{10c_3r^{\alpha+1}}{9c_0\alpha}\le\frac{10c_3\ve^{\alpha+1}}{9c_0\alpha}
\end{equation}
and
\begin{align}\label{v-w-integral-estimate3}
r\left|\int_r^{\ve}\frac{\rho^{-\alpha}v(\rho)^2}{w(\rho)}\,d\rho\right|
\le&\frac{10c_3^2r}{9c_0}\int_r^{\ve}\rho^{\alpha-2}\,d\rho
\le\left\{\begin{aligned}
&\frac{10c_3^2r(r^{\alpha-1}+\ve^{\alpha-1})}{9c_0|\alpha-1|}\quad\,\mbox{ if }n \ne 4\\
&\frac{10c_3^2r|\log r|}{9c_0}\qquad\qquad\mbox{if }n=4
\end{aligned}\right.\notag\\
\le&\left\{\begin{aligned}
&\frac{20c_3^2\3^{\alpha}}{9c_0|\alpha -1|}\quad\mbox{ if }n \ne 4\\
&\frac{10c_3^2c_4\ve^{1/2}}{9c_0}\quad\mbox{ if }n=4.
\end{aligned}\right.
\end{align}
Let
\begin{equation*}
c_5=\left\{\begin{aligned}
&\frac{4n(c_3+c_3^2)}{c_0|\alpha-1|}+\frac{c_3|\lambda|}{c_0\alpha}\qquad\mbox{ if }n \ne 4\\
&\frac{4nc_4(c_3+c_3^2)}{c_0}+\frac{c_3|\lambda|}{c_0\alpha}
\,\quad\mbox{ if }n=4.
\end{aligned}\right.
\end{equation*}
By \eqref{contraction-map-defn}, \eqref{log-r-bd}, \eqref{v-w-integral-estimate1}, \eqref{v-w-integral-estimate2} and  \eqref{v-w-integral-estimate3}, 
\begin{align}\label{phi2-norm-bd1}
&r^{1-\alpha}\left|\Phi_2(w,v)(r)+c_2r^{\alpha-1}\right|\notag\\
\le&|c_1|r+\frac{\alpha |\lambda|}{2}r|\log r|
+\frac{(n-1)r}{2}\left|\int_r^{\ve}\frac{\rho^{-1}v(\rho)}{w(\rho)}\,d\rho\right|+\frac{|\lambda|r}{2}\left|\int_0^r\frac{v(\rho)}{w(\rho)}\,d\rho\right|+\frac{r}{2}
\left|\int_r^{\ve}\frac{\rho^{-\alpha}v(\rho)^2}{w(\rho)}\,d\rho\right|
\notag\\
\le&|c_1|\ve+\frac{\alpha |\lambda|c_4}{2}\ve^{1/2}+c_5(\ve^{\alpha}+\ve^{1/2})\quad\forall 0<r\le\ve.
\end{align}
Let
\begin{equation*}
\ve_2=\min \left(\ve_1,\frac{c_0}{30(|c_1|+1)},\left(\frac{c_0}{30c_5}\right)^{\frac{1}{\alpha}},\frac{c_0^2}{900(\alpha|\lambda|c_4+c_5)^2}\right)
\end{equation*}
and $\ve\in (0,\ve_2)$. Then by \eqref{phi2-norm-bd1},
\begin{align}\label{phi2-bd}
&r^{1-\alpha}\left|\Phi_2(w,v)(r)+c_2r^{\alpha-1}\right|\le c_0/10\quad\forall 0<r\le\ve\notag\\
\Rightarrow\quad&\|r^{1-\alpha}(\Phi_2(w,v)(r)+c_2r^{\alpha-1})\|_{L^{\infty}\left([0, \3]\right)}\le c_0/10.
\end{align}
By \eqref{phi1-bd} and \eqref{phi2-bd}, 
\begin{equation}\label{phi-bd}
\|\Phi(w,v)-(c_0,-c_2r^{\alpha-1})\|_{\cX_\3}\le c_0/10.
\end{equation}
Hence $\Phi(\cD_{\3})\subset \cD_{\3}$. Let $(w_1,v_1), (w_2,v_2)\in \cD_{\3}$, $0<\3<\3_2$, $\delta_1=\|(w_1,v_1)-(w_2,v_2)\|_{\cX_\3}$.
Then
\begin{equation}\label{w12-range}
\frac{9c_0}{10}\le w_i(r)\le \frac{11c_0}{10}\quad\forall 0<r\le\ve, i=1,2
\end{equation} 
and
\begin{equation}\label{v12-bd}
|v_i(r)|\le c_3r^{\alpha -1}\quad\forall 0<r\le\ve, i=1,2.
\end{equation}
Now
\begin{align}\label{phi1-difference}
|\Phi_1(w_1,v_1)(r)-\Phi_1(w_2,v_2)(r)|\le&\int_0^r|v_1(\rho)-v_2(\rho)|\,d\rho\notag\\
\le&\|r^{1-\alpha}(v_1(r)-v_2(r))\|_{L^{\infty}\left([0, \3]\right)}\int_0^r\rho^{\alpha-1}\,d\rho\quad\forall 0<r\le\ve\notag\\
\le&(\delta_1/\alpha)\ve^{\alpha}\quad\forall 0<r\le\ve
\end{align}
and by \eqref{log-r-bd}, \eqref{w12-range} and \eqref{v12-bd}, for any $0<r\le\ve$,
\begin{align}\label{v-w-integral-difference-estimate1}
&r\left|\int_r^{\ve}\frac{\rho^{-1}v_1(\rho)}{w_1(\rho)}\,d\rho-\int_r^{\ve}\frac{\rho^{-1}v_2(\rho)}{w_2(\rho)}\,d\rho\right|\notag\\
\le& r\int_r^{\ve}\frac{\rho^{-1}|v_1(\rho)-v_2(\rho)|}{w_1(\rho)}\,d\rho
+r\int_r^{\ve}\rho^{-1}|v_2(\rho)|\left|\frac{1}{w_1(\rho)}-\frac{1}{w_2(\rho)}\right|\,d\rho\notag\\
\le&\frac{10\|\rho^{1-\alpha}|v_1-v_2|(\rho)\|_{L^{\infty}\left([0, \3]\right)}r}{9c_0}\int_r^{\ve}\rho^{\alpha-2}\,d\rho
+\frac{100c_3\|w_1-w_2\|_{L^{\infty}\left([0, \3]\right)}r}{81c_0^2}
\int_r^{\ve}\rho^{\alpha-2}\,d\rho
\notag\\
\le&\left\{\begin{aligned}
&\left(\frac{10}{9c_0|\alpha -1|}+\frac{100c_3}{81c_0^2|\alpha -1|}\right)\delta_1r(r^{\alpha-1}+\3^{\alpha-1})\quad\mbox{ if }n\ne 4\\
&\left(\frac{10}{9c_0}+\frac{100c_3}{81c_0^2}\right)\delta_1r|\log r|\qquad\qquad\qquad\qquad\,\mbox{ if }n=4
\end{aligned}\right.\notag\\
\le&\left\{\begin{aligned}
&\frac{c_6\delta_1\3^{\alpha}}{|\alpha -1|}\qquad\mbox{ if }n\ne 4\\
&c_4c_6\delta_1\3^{1/2}\quad\mbox{ if }n=4
\end{aligned}\right.
\end{align}
where 
\begin{equation*}
c_6=\frac{20}{9c_0}+\frac{200c_3}{81c_0^2},
\end{equation*}
\begin{align}\label{v-w-integral-difference-estimate2}
&r\left|\int_0^r\frac{v_1(\rho)}{w_1(\rho)}\,d\rho-\int_0^r\frac{v_2(\rho)}{w_2(\rho)}\,d\rho\right|\notag\\
\le& r\int_0^r\frac{|v_1(\rho)-v_2(\rho)|}{w_1(\rho)}\,d\rho
+r\int_0^r|v_2(\rho)|\left|\frac{1}{w_1(\rho)}-\frac{1}{w_2(\rho)}\right|\,d\rho\notag\\
\le&\frac{10\|\rho^{1-\alpha}|v_1-v_2|(\rho)\|_{L^{\infty}\left([0, \3]\right)}r}{9c_0}\int_0^r\rho^{\alpha-1}\,d\rho
+\frac{100c_3\|w_1-w_2\|_{L^{\infty}\left([0, \3]\right)}r}{81c_0^2}
\int_0^r\rho^{\alpha-1}\,d\rho\notag\\
=&\left(\frac{10}{9c_0\alpha}+\frac{100c_3}{81c_0^2\alpha}\right)\delta_1r^{\alpha+1}
\le  \frac{c_6\delta_1\3^{\alpha+1}}{\alpha}
\end{align}
and
\begin{align}\label{v-w-integral-difference-estimate3}
&r\left|\int_r^{\ve}\frac{\rho^{-\alpha}v_1(\rho)^2}{w_1(\rho)}\,d\rho-\int_r^{\ve}\frac{\rho^{-\alpha}v_2(\rho)^2}{w_2(\rho)}\,d\rho\right|\notag\\
\le& r\int_r^{\ve}\frac{\rho^{-\alpha}|v_1(\rho)-v_2(\rho)|(|v_1(\rho)|+|v_2(\rho)|)}{w_1(\rho)}\,d\rho
+r\int_r^{\ve}\rho^{-\alpha}|v_2(\rho)|^2\left|\frac{1}{w_1(\rho)}-\frac{1}{w_2(\rho)}\right|\,d\rho\notag\\
\le&\frac{20c_3\|\rho^{1-\alpha}|v_1-v_2|(\rho)\|_{L^{\infty}\left([0, \3]\right)}r}{9c_0}\int_r^{\ve}\rho^{\alpha-2}\,d\rho
+\frac{100c_3^2\|w_1-w_2\|_{L^{\infty}\left([0, \3]\right)}r}{81c_0^2}
\int_r^{\ve}\rho^{\alpha-2}\,d\rho
\notag\\
\le&\left\{\begin{aligned}
&\left(\frac{20c_3}{9c_0|\alpha -1|}+\frac{100c_3^2}{81c_0^2|\alpha -1|}\right)\delta_1r(r^{\alpha-1}+\3^{\alpha-1})\quad\mbox{ if }n\ne 4\\
&\left(\frac{20c_3}{9c_0}+\frac{100c_3^2}{81c_0^2}\right)\delta_1r|\log r|\qquad\qquad\qquad\qquad\,\mbox{if }n=4
\end{aligned}\right.\notag\\
\le&\left\{\begin{aligned}
&\frac{2c_3c_6\delta_1\3^{\alpha}}{|\alpha -1|}\qquad\,\,\mbox{ if }n\ne 4\\
&2c_3c_4c_6\delta_1\3^{1/2}\quad\mbox{ if }n=4.
\end{aligned}\right.
\end{align}
By \eqref{contraction-map-defn}, \eqref{v-w-integral-difference-estimate1}, \eqref{v-w-integral-difference-estimate2} and \eqref{v-w-integral-difference-estimate3},
\begin{equation}\label{phi2-difference}
r^{1-\alpha}|\Phi_2(w_1,v_1)(r)-\Phi_2(w_2,v_2)(r)|\le c_7\delta_1(\3^{\alpha}+\ve^{1/2})\quad\forall 0<r\le\ve
\end{equation}
where 
\begin{equation*}
c_7=\left\{\begin{aligned}
&c_6\left(\frac{n(1+c_3)}{|\alpha -1|}+\frac{|\lambda|}{\alpha}\right)\quad\,\,\,\mbox{ if }n \ne 4\\
&c_6\left(nc_4(1+c_3)+\frac{|\lambda|}{\alpha}\right)\quad\mbox{ if }n=4.
\end{aligned}\right.
\end{equation*}
Let
\begin{equation*}
\ve_3=\min\left(\ve_2,(\alpha/6)^{1/\alpha},(6c_7)^{-1/\alpha}, (6c_7)^{-2}\right)
\end{equation*}
and $0<\3<\ve_3$.
By \eqref{phi1-difference} and \eqref{phi2-difference},
\begin{equation}\label{phi-difference}
\|\Phi(w_1,v_1)-\Phi(w_2,v_2)\|_{\cX_\3}\le\delta_1/2.
\end{equation}
Hence $\Phi$ is a contraction map on  $\cD_{\3}$. Therefore by the contraction map theorem there exists a unique fixed point $(w,v)=\Phi(w,v)$ in $\cD_{\3}$. Thus
\begin{equation}\label{w-v-eqn}
\left\{\begin{aligned}
w(r)=&c_0+\int_0^r v(\rho)\,d\rho,\\
v(r)=&-c_2r^{\alpha-1}+c_1r^{\alpha}+\frac{\alpha\lambda}{2}r^{\alpha}\log r
+r^{\alpha}\left\{\frac{(n-1)}{2}\int_r^{\ve}\frac{\rho^{-1}v(\rho)}{w(\rho)}\,d\rho
-\frac{\lambda}{2}\int_0^r\frac{v(\rho)}{w(\rho)}\,d\rho\right.\\
&\quad \left.-\frac{1}{2}
\int_r^{\ve}\frac{\rho^{-\alpha}v(\rho)^2}{w(\rho)}\,d\rho\right\}.
\end{aligned}\right.
\end{equation}
By \eqref{w-v-eqn} $v(r)=w_r(r)$ for any $0<r\le\3$ and $w\in C^2((0,\3];\Rr^+)\cap C([0,\ve],\Rr^+)$ satisfies \eqref{w0=c0}  and \eqref{w-eqn=w-integral-eqn}.  Hence $w$ satisfies \eqref{wrr-eqn}.
By \eqref{w-eqn=w-integral-eqn}, \eqref{v-w-integral-estimate1}, \eqref{v-w-integral-estimate2}, and \eqref{v-w-integral-estimate3}, we get \eqref{wr-initial-value} and the proposition follows.
\end{proof}

By an argument similar to the proof of Proposition \ref{w-local-existence-prop} we have the following result.

\begin{prop}\label{w-local-existence-prop2}
Let $n\in\Z^+$, $n>4$, $\alpha=\sqrt{n}-1$, $\lambda, c_1\in\Rr$, $c_0>0$  and let $c_2$ be given by \eqref{c2-defn}.  Then there exists a constant $0<\3<1$ such that \eqref{wrr-eqn} has a unique solution $w\in C^2((0,\3];\Rr^+)\cap C([0,\3];\Rr^+)$ in $(0,\3]$ which satisfies   \eqref{w0=c0} and 
\begin{align}\label{w-eqn=w-integral-eqn2}
w_r(r)=&-c_2r^{\alpha-1}+c_1r^{\alpha}+\frac{\alpha\lambda}{2}r^{\alpha}\log r
+r^{\alpha}\left\{-\frac{(n-1)}{2}\int_0^r\frac{\rho^{-1}w_r(\rho)}{w(\rho)}\,d\rho-\frac{\lambda}{2}\int_0^r\frac{w_r(\rho)}{w(\rho)}\,d\rho\right.\notag\\
&\quad \left.+\frac{1}{2}
\int_0^r\frac{\rho^{-\alpha}w_r(\rho)^2}{w(\rho)}\,d\rho\right\}\qquad\qquad\forall 0<r\le\3.
\end{align}
Moreover \eqref{wr-initial-value} holds.
\end{prop} 

\begin{cor}\label{h-local-existence-cor1}
Let $2\le n\in\Z^+$, $\alpha=\sqrt{n}-1$, $\lambda, c_1\in\Rr$, $c_0>0$  and let $c_2$ be given by \eqref{c2-defn}.  Then there exists a constant $0<\3<1$ such that \eqref{h-eqn} has infinitely many solutions $h\in C^2((0,\3])$ in $(0,\3]$ which satisfies \eqref{h-origin-blow-up-rate}. Moreover \eqref{h-eqn} has a unique solution $h\in C^2((0,\3])$ in $(0,\3]$ which satisfies \eqref{h-origin-blow-up-rate} and \eqref{w-eqn=w-integral-eqn} with  $w$ being given by \eqref{w-defn}. Moreover \eqref{wr-initial-value} holds.
\end{cor} 

\begin{cor}\label{h-local-existence-cor2}
Let $n\in\Z^+$, $n>4$, $\alpha=\sqrt{n}-1$, $\lambda, c_1\in\Rr$, $c_0>0$  and let $c_2$ be given by \eqref{c2-defn}.    Then  there exists a constant $0<\3<1$ such that \eqref{h-eqn} has a unique solution $h\in C^2((0,\3])$ in $(0,\3]$ which satisfies \eqref{h-origin-blow-up-rate} and \eqref{w-eqn=w-integral-eqn2} with  $w$ being given by \eqref{w-defn}. Moreover \eqref{wr-initial-value} holds.

\end{cor} 

\begin{prop}\label{h-asymptotic-behaviour-near-0-prop1}
Let $n\in\Z^+$, $n>4$, $\alpha=\sqrt{n}-1$, $\lambda, c_1\in\Rr$, $c_0>0$  and let $c_2$ be given by \eqref{c2-defn}.   Then there exists a constant $0<\3<1$ such that \eqref{h-eqn} has a unique solution $h\in C^2((0,\3])$ in $(0,\3]$ which satisfies \eqref{h-origin-blow-up-rate} and \eqref{h-asymptotic-near-origin1} for some constant $0<\delta_0<\ve$.
Moreover \eqref{hr-asymptotic-near-origin1} and \eqref{wr-initial-value} holds with  $w$ being given by \eqref{w-defn}.
\end{prop}
\begin{proof}
Since $n>4$, $\alpha>1$ and $c_2>0$.  Let $w$ be given by \eqref{w-defn}.
By Corollary \ref{h-local-existence-cor2} there exists a constant $0<\3<1$ such that \eqref{h-eqn} has a unique solution $h\in C^2((0,\3])$ in $(0,\3]$ which satisfies \eqref{h-origin-blow-up-rate}, \eqref{wr-initial-value}  and \eqref{w-eqn=w-integral-eqn2}.  Let
\begin{equation}\label{delta1-defn}
0<\delta_1<\min\left(\3,\frac{|c_2|}{2c_0},\frac{c_2^2}{2c_0}\right).
\end{equation}
By \eqref{h-origin-blow-up-rate} and \eqref{wr-initial-value}  there exist constants $\delta_0\in (0,\ve)$ and $c_8>0$ such that
\begin{equation}\label{wr-w-ratio-upper-lower-bds1}
-\frac{c_2}{c_0}-\delta_1\le \frac{r^{1-\alpha}w_r(r)}{w(r)}
\le -\frac{c_2}{c_0}+\delta_1\quad\forall 0<r\le\delta_0
\end{equation}
and
\begin{equation}\label{wr-w-ratio-upper-lower-bds2}
\frac{c_2^2}{c_0}-\delta_1\le \frac{(r^{1-\alpha}w_r(r))^2}{w(r)}
\le \frac{c_2^2}{c_0}+\delta_1\quad\forall 0<r\le\delta_0
\end{equation}
holds.
Then by \eqref{wr-w-ratio-upper-lower-bds1} and \eqref{wr-w-ratio-upper-lower-bds2},
\begin{align}\label{integrals-bd2}
&-\frac{(n-1)}{2}\int_0^r\frac{\rho^{-1}w_r(\rho)}{w(\rho)}\,d\rho-\frac{\lambda}{2}\int_0^r\frac{w_r(\rho)}{w(\rho)}\,d\rho+\frac{1}{2}\int_0^r\frac{\rho^{-\alpha}w_r(\rho)^2}{w(\rho)}\,d\rho\notag\\
\le&\frac{(n-1)}{2}\left(\frac{c_2}{c_0}+\delta_1\right)\int_0^r\rho^{\alpha-2}\,d\rho
+\frac{\lambda}{2}\left(\frac{c_2}{c_0}+\mbox{sign}\,(\lambda)\delta_1\right)\int_0^r\rho^{\alpha-1}\,d\rho+\frac{1}{2}\left(\frac{c_2^2}{c_0}+\delta_1\right)\int_0^r\rho^{\alpha-2}\,d\rho
\notag\\
\le& \frac{c_2(n-1+c_2)}{2c_0(\alpha -1)}r^{\alpha -1}+\frac{\lambda c_2}{2c_0\alpha}r^{\alpha}+\frac{n\delta_1}{2(\alpha -1)}r^{\alpha -1}+\frac{|\lambda|\delta_1}{2\alpha}r^{\alpha}\quad\forall 0<r\le\delta_0
\end{align}
and
\begin{align}\label{integrals-bd2a}
&-\frac{(n-1)}{2}\int_0^r\frac{\rho^{-1}w_r(\rho)}{w(\rho)}\,d\rho-\frac{\lambda}{2}\int_0^r\frac{w_r(\rho)}{w(\rho)}\,d\rho+\frac{1}{2}\int_0^r\frac{\rho^{-\alpha}w_r(\rho)^2}{w(\rho)}\,d\rho\notag\\
\ge&\frac{(n-1)}{2}\left(\frac{c_2}{c_0}-\delta_1\right)\int_0^r\rho^{\alpha-2}\,d\rho
+\frac{\lambda}{2}\left(\frac{c_2}{c_0}-\mbox{sign}\,(\lambda)\delta_1\right)\int_0^r\rho^{\alpha-1}\,d\rho+\frac{1}{2}\left(\frac{c_2^2}{c_0}-\delta_1\right)\int_0^r\rho^{\alpha-2}\,d\rho
\notag\\
\ge& \frac{c_2(n-1+c_2)}{2c_0(\alpha -1)}r^{\alpha -1}+\frac{\lambda c_2}{2c_0\alpha}r^{\alpha}-\frac{n\delta_1}{2(\alpha -1)}r^{\alpha -1}-\frac{|\lambda|\delta_1}{2\alpha}r^{\alpha}\quad\forall 0<r\le\delta_0.
\end{align}
Hence by \eqref{h-origin-blow-up-rate}, \eqref{w-eqn=w-integral-eqn2}, \eqref{integrals-bd2} and \eqref{integrals-bd2a},
\begin{align}
&w_r(r)=-c_2r^{\alpha-1}+c_1r^{\alpha}+\frac{\alpha\lambda}{2}r^{\alpha}\log r
+\frac{c_2(n-1+c_2)}{2c_0(\alpha -1)}r^{2\alpha -1}+o(1)(r^{2\alpha-1})\qquad\qquad\forall 0<r\le\delta_0\label{wr-eqn10}\\
\Rightarrow\quad&r^{\alpha}h(r)=c_0-\frac{c_2}{\alpha}r^{\alpha}+\left(\frac{c_1}{\alpha+1}-\frac{\alpha\lambda}{2(\alpha +1)^2}\right)r^{\alpha +1}+\frac{\alpha\lambda}{2\alpha +2}r^{\alpha +1}\log r+\frac{c_2(n-1+c_2)}{4c_0\alpha (\alpha -1)}r^{2\alpha}\notag\\
&\qquad\qquad +o(1)r^{2\alpha}\quad\forall 0<r\le\delta_0\notag
\end{align}
 and \eqref{h-asymptotic-near-origin1} follows. Since
\begin{equation}\label{wr-identity}
w_r(r)=\alpha r^{\alpha-1}h(r)+r^{\alpha}h_r(r)\quad\forall r>0,
\end{equation}
by \eqref{h-asymptotic-near-origin1} and \eqref{wr-eqn10} we get \eqref{hr-asymptotic-near-origin1}.

Suppose $h_1\in C^2((0,\3))$ is another solution of 
\eqref{h-eqn} which satisfies \eqref{h-origin-blow-up-rate} and 
\begin{align}\label{h1-asymptotic-near-origin}
h_1(r)=&\frac{1}{r^{\alpha}}\left\{c_0-\frac{c_2}{\alpha}r^{\alpha}+\left(\frac{c_1}{\alpha+1}-\frac{\alpha\lambda}{2(\alpha +1)^2}\right)r^{\alpha +1}+\frac{\alpha\lambda}{2\alpha +2}r^{\alpha +1}\log r+\frac{c_2(n-1+c_2)}{4c_0\alpha (\alpha -1)}r^{2\alpha}\right.\notag\\
&\quad +\left.o(1)r^{2\alpha}\right\}\quad\forall 0<r\le \delta_0.
\end{align}
Let $w_1(r)=r^{\alpha}h_1(r)$. Then $w_1$ satisfies \eqref{wrr-eqn2}. Integrating the equation \eqref{wrr-eqn2} for $w_1$ over $(0,r)$ we get
\begin{align}\label{w1-eqn=w-integral-eqn2}
w_{1,r}(r)=&-c_2r^{\alpha-1}+c_1'r^{\alpha}+\frac{\alpha\lambda}{2}r^{\alpha}\log r
+r^{\alpha}\left\{-\frac{(n-1)}{2}\int_0^r\frac{\rho^{-1}w_{1,r}(\rho)}{w_1(\rho)}\,d\rho-\frac{\lambda}{2}\int_0^r\frac{w_{1,r}(\rho)}{w_1(\rho)}\,d\rho\right.\notag\\
&\quad \left.+\frac{1}{2}
\int_0^r\frac{\rho^{-\alpha}w_{1,r}(\rho)^2}{w_1(\rho)}\,d\rho\right\}\qquad\qquad\forall 0<r\le\3
\end{align}
for some constant $c_1'\in\Rr$. By \eqref{w1-eqn=w-integral-eqn2} and a similar argument as before 
we get
\begin{align}\label{h1-asymptotic-near-origin2}
h_1(r)=&\frac{1}{r^{\alpha}}\left\{c_0-\frac{c_2}{\alpha}r^{\alpha}+\left(\frac{c_1'}{\alpha+1}-\frac{\alpha\lambda}{2(\alpha +1)^2}\right)r^{\alpha +1}+\frac{\alpha\lambda}{2\alpha +2}r^{\alpha +1}\log r+\frac{c_2(n-1+c_2)}{4c_0\alpha (\alpha -1)}r^{2\alpha}\right.\notag\\
&\quad +\left.o(1)r^{2\alpha}\right\}\quad\forall 0<r\le \delta_0.
\end{align}
By \eqref{h1-asymptotic-near-origin} and \eqref{h1-asymptotic-near-origin2},
\begin{equation*}
c_1-c_1'=o(1)(r^{\alpha-1})\quad\forall 0<r\le \delta_0\quad\Rightarrow \quad c_1=c_1'\quad\mbox{ as }r\to 0^+.
\end{equation*}
Hence both $w$ and $w_1$ satisfies \eqref{w-eqn=w-integral-eqn2}. Then by Proposition \ref{w-local-existence-prop2}, $w\equiv w_1$ on $[0,\3]$. Thus $h=h_1$ on $[0,\3]$ and the solution $h$ is unique.
\end{proof}

\begin{prop}\label{h-asymptotic-behaviour-near-0-prop2}
Let $n\in\{2,3,4\}$, $\alpha=\sqrt{n}-1$, $\lambda, c_1\in\Rr$, $c_0>0$  and let $c_2$ be given by \eqref{c2-defn}.    Let $0<\3<1$ and $h\in C^2((0,\3])$ be the unique solution  of \eqref{h-eqn} in $(0,\3]$ given by Corollary \ref{h-local-existence-cor1} which satisfies \eqref{h-origin-blow-up-rate}, \eqref{w-eqn=w-integral-eqn} and \eqref{wr-initial-value} with  $w$ being given by \eqref{w-defn}. Then there exists a constant $0<\delta_0<\ve$ such that
\eqref{h-asymptotic-near-origin2} and \eqref{hr-asymptotic-near-origin2} holds.
\end{prop} 
\begin{proof}
Note that $\alpha<1$, $1/(1-\alpha)\le 4$ and $c_2<0$ when  $n=2$ or $3$ and $\alpha=1$, $c_2=0$, when $n=4$.  Let $\delta_1$ satisfy \eqref{delta1-defn} when $n=2,3$ and $0<\delta_1<\ve$ when $n=4$.
By \eqref{h-origin-blow-up-rate} and \eqref{wr-initial-value}  there exists a constant $0<\delta_0<\ve$  such that \eqref{wr-w-ratio-upper-lower-bds1} and \eqref{wr-w-ratio-upper-lower-bds2} holds.
Hence by \eqref{wr-w-ratio-upper-lower-bds1} and \eqref{wr-w-ratio-upper-lower-bds2}, for any $0<r<\delta_0$,
\begin{align}\label{integrals-bd1}
&\frac{(n-1)}{2}\int_r^{\ve}\frac{\rho^{-1}w_r(\rho)}{w(\rho)}\,d\rho-\frac{\lambda}{2}\int_0^r\frac{w_r(\rho)}{w(\rho)}\,d\rho-\frac{1}{2}\int_r^{\ve}\frac{\rho^{-\alpha}w_r(\rho)^2}{w(\rho)}\,d\rho\notag\\
\le&\frac{(n-1)}{2}\left(-\frac{c_2}{c_0}+\delta_1\right)\int_r^{\delta_0}\rho^{\alpha-2}\,d\rho
+\frac{\lambda}{2}\left(\frac{c_2}{c_0}+\mbox{sign}\,(\lambda)\delta_1\right)\int_0^r\rho^{\alpha-1}\,d\rho
-\frac{1}{2}\left(\frac{c_2^2}{c_0}-\delta_1\right)\int_r^{\delta_0}\rho^{\alpha-2}\,d\rho\notag\\
&\quad +c_8\notag\\
\le&\left\{\begin{aligned}
&-\frac{c_2(c_2+n-1)}{2c_0(1-\alpha)}(r^{\alpha -1}-\delta_0^{\alpha -1})+\frac{\lambda c_2}{2c_0\alpha}r^{\alpha}+\frac{n\delta_1}{2(1-\alpha)}(r^{\alpha -1}-\delta_0^{\alpha -1})+\frac{|\lambda|\delta_1}{2\alpha}r^{\alpha}+c_8\quad\mbox{ if }n=2,3\\
&\frac{n\delta_1}{2}(\log\delta_0-\log r)+\frac{|\lambda|\delta_1}{2}r+c_8\qquad\qquad\qquad\qquad\qquad\qquad\qquad\qquad\qquad\,\,\mbox{ if } n=4
\end{aligned}\right.
\end{align}
and
\begin{align}\label{integrals-bd1a}
&\frac{(n-1)}{2}\int_r^{\ve}\frac{\rho^{-1}w_r(\rho)}{w(\rho)}\,d\rho-\frac{\lambda}{2}\int_0^r\frac{w_r(\rho)}{w(\rho)}\,d\rho-\frac{1}{2}\int_r^{\ve}\frac{\rho^{-\alpha}w_r(\rho)^2}{w(\rho)}\,d\rho\notag\\
\ge&\frac{(n-1)}{2}\left(-\frac{c_2}{c_0}-\delta_1\right)\int_r^{\delta_0}\rho^{\alpha-2}\,d\rho
+\frac{\lambda}{2}\left(\frac{c_2}{c_0}-\mbox{sign}\,(\lambda)\delta_1\right)\int_0^r\rho^{\alpha-1}\,d\rho
-\frac{1}{2}\left(\frac{c_2^2}{c_0}+\delta_1\right)\int_r^{\delta_0}\rho^{\alpha-2}\,d\rho\notag\\
&\quad +c_8\notag\\
\ge&\left\{\begin{aligned}
&-\frac{c_2(c_2+n-1)}{2c_0(1-\alpha)}(r^{\alpha -1}-\delta_0^{\alpha -1})+\frac{\lambda c_2}{2c_0\alpha}r^{\alpha}-\frac{n\delta_1}{2(1-\alpha)}(r^{\alpha -1}-\delta_0^{\alpha -1})-\frac{|\lambda|\delta_1}{2\alpha}r^{\alpha}+c_8\quad\mbox{ if }n=2,3\\
&-\frac{n\delta_1}{2}(\log\delta_0-\log r)-\frac{|\lambda|\delta_1}{2}r+c_8\qquad\qquad\qquad\qquad\qquad\qquad\qquad\qquad\quad\mbox{ if } n=4
\end{aligned}\right.
\end{align}
where 
\begin{equation*}
c_8=\frac{(n-1)}{2}\int_{\delta_0}^{\ve}\frac{\rho^{-1}w_r(\rho)}{w(\rho)}\,d\rho
-\frac{1}{2}\int_{\delta_0}^{\ve}\frac{\rho^{-\alpha}w_r(\rho)^2}{w(\rho)}\,d\rho.
\end{equation*}
Thus by \eqref{h-origin-blow-up-rate}, \eqref{w-eqn=w-integral-eqn}, \eqref{integrals-bd1} and\eqref{integrals-bd1a},
\begin{align}
&w_r(r)=\left\{\begin{aligned}
&-c_2r^{\alpha-1}-\frac{c_2(c_2+n-1)}{2c_0(1-\alpha)}r^{2\alpha-1}+o(1)(r^{2\alpha-1})\quad\forall 0<r\le\delta_0\quad\mbox{ if }n=2,3\\
&\frac{\lambda}{2}r\log r+o(1)(r|\log r|)\qquad\qquad\qquad\qquad\quad\forall 0<r\le\delta_0\quad\mbox{ if }n=4
\end{aligned}\right.\label{wr-eqn11}\\
\Rightarrow\quad&w(r)=r^{\alpha}h(r)=\left\{\begin{aligned}
&c_0-\frac{c_2}{\alpha}r^{\alpha}-\frac{c_2(c_2+n-1)}{4c_0\alpha (1-\alpha)}r^{2\alpha}+o(1)r^{2\alpha}\quad\forall 0<r\le\delta_0
\quad\mbox{ if }n=2,3\\
&c_0+\frac{\lambda}{4}r^2\log r+o(1)r^2|\log r|\qquad\qquad\quad\forall 0<r\le\delta_0\quad\mbox{ if }n=4
\end{aligned}\right.\label{w-eqn11}
\end{align}
and \eqref{h-asymptotic-near-origin2} follows. By \eqref{h-asymptotic-near-origin2},  \eqref{wr-identity} and \eqref{wr-eqn11} we get \eqref{hr-asymptotic-near-origin2} and the proposition follows.

\end{proof}

\section{Global existence and uniqueness of singular solutions}
\setcounter{equation}{0}
\setcounter{thm}{0}

In this section we will use a modification of the technique  of S.Y.~Hsu \cite{H} to prove the  
global existence of infinitely many singular solutions of \eqref{h-eqn}, \eqref{h-origin-blow-up-rate}, in $(0,\infty)$.
We will also prove the uniqueness of the global singular solution of such equation in terms of its asymptotic behaviour near the origin.  

\begin{lem}\label{h-hr-eqn}
Let $2\le n\in\Z^+$, $\lambda\in\mathbb{R}$ and $L>0$. Suppose $h\in C^2((0,L))$ satisfies \eqref{h-eqn} in $(0,L)$. Then  
\begin{equation}\label{h-derivative-integral-formula}
h_r(r_1)=\frac{n-1}{r_1}+\lambda+\sqrt{\frac{h(r_1)}{h(r_2)}}\left(h_r(r_2)-\frac{n-1}{r_2}-\lambda\right)+\frac{(n-1)\sqrt{h(r_1)}}{2}\int_{r_2}^{r_1}\frac{h(\rho)+1}{\rho^2\sqrt{h(\rho)}}\,d\rho
\end{equation} 
holds for any $0<r_2<r_1<L$.
\end{lem}
\begin{proof}
By \eqref{h-eqn},
\begin{align*}
(h^{-1/2}h_r)_r=&\frac{(n-1)(h-1)}{2r^2h^{1/2}}-\frac{(n-1+\lambda r)h_r}{2rh^{3/2}}\quad\forall r>0\\
\Rightarrow\qquad h_r(r_1)=&\sqrt{h(r_1)}\left\{\frac{h_r(r_2)}{\sqrt{h(r_2)}}+\frac{(n-1)}{2}\int_{r_2}^{r_1}\frac{h(\rho)-1}{\rho^2\sqrt{h(\rho)}}\,d\rho-\int_{r_2}^{r_1}\frac{(n-1+\lambda\rho)h_r(\rho)}{2\rho h(\rho)^{3/2}}\,d\rho\right\}\\
=&\sqrt{h(r_1)}\left\{\frac{h_r(r_2)}{\sqrt{h(r_2)}}+\frac{(n-1)}{2}\int_{r_2}^{r_1}\frac{h(\rho)-1}{\rho^2\sqrt{h(\rho)}}\,d\rho+\left(\frac{n-1}{r_1}+\lambda\right)\frac{1}{\sqrt{h(r_1)}}\right.\\
&\qquad\left.-\left(\frac{n-1}{r_2}+\lambda\right)\frac{1}{\sqrt{h(r_2)}}+(n-1)\int_{r_2}^{r_1}\frac{d\rho}{\rho^2\sqrt{h(\rho)}}\right\}\quad\forall 0<r_2<r_1<L
\end{align*}
and \eqref{h-derivative-integral-formula} follows.

\end{proof}

We now observe that by an argument similar to the proof of Lemma 2.3, Lemma 2.4, Lemma 2.5 and Lemma 2.6 of \cite{H} but with \eqref{h-eqn} and \eqref{h-derivative-integral-formula} replacing (1.6) and (2.25) of \cite{H} in the proof there we have the following results.

\begin{lem}(cf. Lemma 2.3 and Lemma 2.4 of \cite{H})\label{h-bd-lem}
Let $2\le n\in\Z^+$ and $\lambda\in\mathbb{R}$. Suppose $h\in C^2((0,L))$ satisfies \eqref{h-eqn}
in $(0,L)$ for some constant $L\in (0,\infty)$  such that $L<-(n-1)/\lambda$ if $\lambda<0$. Then there exist  constants $C_2>C_1>0$ such that
\begin{equation}\label{h-uniform-bd}
C_1\le h(r)\le C_2\quad\forall L/2\le r\le L.
\end{equation}
\end{lem}

\begin{lem}(cf. Lemma 2.5 and Lemma 2.6 of \cite{H})\label{h-derivative-bd-lem}
Let $2\le n\in\Z^+$ and $\lambda\in\mathbb{R}$. Suppose $h\in C^2((0,L))$ satisfies \eqref{h-eqn}
in $(0,L)$ for some constant $L\in (0,\infty)$  such that $L<-(n-1)/\lambda$ if $\lambda<0$. Then there exist constants $C_4>C_3$ such that
\begin{equation}\label{h-derivative-uniform-bd}
C_3\le h_r(r)\le C_4\quad\forall L/2\le r\le L.
\end{equation}  
\end{lem}

We next observe that by standard ODE theory we have the following result.

\begin{lem}(cf. Lemma 2.7 of \cite{H})\label{h-extension-lem}
Let $2\le n\in\Z^+$, $\lambda\in\Rr$,  $L>0$, $b_0\in (C_1,C_2)$, $b_1\in (C_4,C_3)$ for some constants $C_2>C_1>0$ and $C_3>C_4$. Then there exists a constant $0<\delta_1<L/4$ depending only on $C_1, C_2, C_3, C_4$ such that for any $r_0\in (L/2,L)$ \eqref{h-eqn} has a unique solution $\4{h}\in C^2((r_0-\delta_1,r_0+\delta_1))$ in $(r_0-\delta_1,r_0+\delta_1)$ which satisfies
\begin{equation}\label{h-tilde-initial-condition}
\4{h}(r_0)=b_0\quad\mbox{ and }\quad \4{h}_r(r_0)=b_1.
\end{equation} 
\end{lem} 

We are now ready for the proof of Theorem \ref{h-existence-thm}.

\noindent{\bf Proof of Theorem \ref{h-existence-thm}}: 
We will use a modification of the proof of Theorem 1.1 of \cite{H} to prove the theorm.
We first observe that by Corollary \ref{h-local-existence-cor1} there exists a constant $0<\3<1$ such that \eqref{h-eqn} has a unique solution $h\in C^2((0,\3])$ in $(0,\3]$ which satisfies \eqref{h-origin-blow-up-rate} and \eqref{w-eqn=w-integral-eqn} with  $w$ being given by \eqref{w-defn}. Moreover \eqref{wr-initial-value} holds. Let $(0,L)$ be the maximal interval of existence of solution $h\in C^2((0,L))$ of \eqref{h-eqn} in $(0,L)$ which satisfies \eqref{h-origin-blow-up-rate}. Suppose $L<\infty$. Then by Lemma \ref{h-bd-lem} and Lemma \ref{h-derivative-bd-lem} there exist constants 
$C_2>C_1>0$ and $C_3>C_4$ such that \eqref{h-uniform-bd} and \eqref{h-derivative-uniform-bd} hold. 

 Then by Lemma \ref{h-extension-lem}  there exists a constant $0<\delta_1<L/4$ depending only on $C_1, C_2, C_3,C_4$ such that for any $r_0\in (L/2,L)$ \eqref{h-eqn} has a unique solution $\4{h}\in C^2((r_0-\delta_1,r_0+\delta_1))$ in $(r_0-\delta_1,r_0+\delta_1)$ which satisfies
\eqref{h-tilde-initial-condition} with $b_0=h(r_0)$ and $b_1=h_r(r_0)$. We now set $r_0=L-(\delta_1/2)$ and extend $h$ to a function on $[0,L+(\delta_1/2))$ by setting
$h(r)=\4{h}(r)$ for any $r\in [L,L+(\delta_1/2))$. Then $h\in C^2((0,L+(\delta_1/2)))$ is a solution of \eqref{h-eqn} in $(0,L+\delta_1)$ which satisfies \eqref{h-origin-blow-up-rate} and \eqref{w-eqn=w-integral-eqn}. This contradicts the choice of $L$. Hence $L=\infty$ and there exists a solution $h\in C^2((0,\infty))$ of \eqref{h-eqn} which satisfies \eqref{h-origin-blow-up-rate} and \eqref{w-eqn=w-integral-eqn}. 

Suppose $h_1\in C^2((0,\infty))$ is another solution of \eqref{h-eqn} which satisfies \eqref{h-origin-blow-up-rate} and \eqref{w-eqn=w-integral-eqn} with $w$ being replaced by $w_1=r^{\alpha}h_1(r)$. Then both $w$ and $w_1$ satisfies \eqref{wrr-eqn2}. Hence both $w$ and $w_1$ satisfies \eqref{wrr-eqn} and \eqref{w0=c0} in $(0,\ve]$. Therefore by Proposition \ref{w-local-existence-prop} $w(r)\equiv w_1(r)$ in $(0,\ve]$. Hence $h(r)=h_1(r)$ in $(0,\ve]$. Then by standard ODE theory $h(r)=h_1(r)$ in $[\ve,\infty)$. Thus $h(r)=h_1(r)$ in $(0,\infty)$ and the solution $h$ is unique.

{\hfill$\square$\vspace{6pt}}

\noindent{\bf Proof of Theorem \ref{h-asymptotic-behaviour-thm1}}:
By Corollary \ref{h-local-existence-cor2} and an argument similar to the proof of Theorem \ref{h-existence-thm} there exists a  unique solution $h\in C^2((0,\infty))$ of \eqref{h-eqn} in $(0,\infty)$ which satisfies \eqref{h-origin-blow-up-rate} and \eqref{w-eqn=w-integral-eqn2} in $(0,\ve)$ with $w$ given by \eqref{w-defn} for some $0<\ve<1$. By \eqref{w-eqn=w-integral-eqn2} and the same argument as the proof of Proposition \ref{h-asymptotic-behaviour-near-0-prop1} we get \eqref{h-asymptotic-near-origin1} and   \eqref{hr-asymptotic-near-origin1} . Suppose $h_1\in C^2((0,\infty))$ is another solution of \eqref{h-eqn} in $(0,\infty)$ which satisfies \eqref{h-origin-blow-up-rate} and \eqref{h-asymptotic-near-origin1}. Then by an argument similar to the proof of Proposition \ref{h-asymptotic-behaviour-near-0-prop1}, $h_1(r)\equiv h(r)$ in $(0,\ve]$. Hence by  standard ODE uniqueness theory $h_1(r)\equiv h(r)$ in $[\ve,\infty)$ and the theorem follows.

{\hfill$\square$\vspace{6pt}}

Finally by Theorem \ref{h-existence-thm} and an argument similar  to the proof of Proposition \ref{h-asymptotic-behaviour-near-0-prop2} we get Theorem \ref{h-asymptotic-behaviour-thm2}.

\noindent{\bf Proof of Theorem \ref{blow-up-rate-at-origin-thm}}:
Without loss of generality we may assume that $\3=2$. Let $w$ be given by \eqref{w-defn} and 
\begin{equation}\label{q-defn}
q(r)=\frac{rh_r(r)}{h(r)}.
\end{equation}
We first claim that there exists a decreasing sequence $\{r_i\}_{i=1}^{\infty}\subset (0,\3)$ such that
\begin{equation}\label{qri-sequence-limit}
\lim_{i\to\infty}q(r_i)=-\alpha.
\end{equation} 
To prove the claim we note that by \eqref{w-defn},
\begin{equation}
w_r(r)=\alpha r^{\alpha-1}h(r)+r^{\alpha}h_r(r)=\frac{w(r)}{r}(\alpha+q(r))\quad\forall 0<r<\ve.
\end{equation}
For any $i\in\Z^+$ by the mean value theorem there exists $r_i\in (1/(2i),1/i)$ such that
\begin{equation}\label{wri-eqn}
w_r(r_i)=2i(w(1/i)-w(1/(2i)).
\end{equation}
By \eqref{h-origin-blow-up-rate}, \eqref{qri-sequence-limit} and \eqref{wri-eqn},
\begin{align*}
&|\alpha+q(r_i)|\le\frac{2ir_i|w(1/i)-w(1/(2i))|}{w(r_i)}\le\frac{2|w(1/i)-w(1/(2i))|}{w(r_i)}\quad\forall i\in\Z^+\\
\Rightarrow\quad&\lim_{i\to\infty}|\alpha+q(r_i)|=0
\end{align*}
and the claim follows. By \eqref{h-eqn} and a direct computation $q$ satisfies
\begin{equation}\label{q-eqn}
q_r(r)+\left(-\frac{1}{r}+\frac{\lambda}{2h(r)}+\frac{n-1}{2rh(r)}\right)q(r)=-\frac{1}{2r}\left(q(r)^2-\frac{(n-1)(h(r)-1)}{h(r)}\right)\quad\forall 0<r<\ve.
\end{equation}
Let
\begin{equation}
F(r)=\mbox{exp}\left(\frac{\lambda}{2}\int_0^r\frac{d\rho}{h(\rho)} +\frac{n-1}{2}\int_0^r\frac{ d\rho}{\rho h(\rho)}\right)\quad\forall 0<r<\ve.
\end{equation}
Then by \eqref{q-eqn},
\begin{align}
&(r^{-1}F(r)q(r))_r=-\frac{F(r)}{2r^2}\left(q(r)^2-\frac{(n-1)(h(r)-1)}{h(r)}\right)\quad\forall 0<r<\ve\notag\notag\\
\Rightarrow\quad&q(r)=\frac{1}{r^{-1}F(r)}(F(1)q(1)+I_1(r))\quad\forall 0<r<\ve.\label{q-integral-represenation}
\end{align}
where
\begin{equation}
I_1(r)=\int_r^1\frac{F(\rho)}{2\rho^2}\left(q(\rho)^2-\frac{(n-1)(h(\rho)-1)}{h(\rho)}\right)\,d\rho\quad\forall 0<r<\ve.
\end{equation}
We now divide the proof into two cases.

\noindent $\underline{\text{\bf Case 1}}$: $\underset{\substack{i\to\infty}}{\limsup}|I_1(r_i)|<\infty$.

\noindent By \eqref{qri-sequence-limit} and \eqref{q-integral-represenation},
\begin{equation*}
-\alpha=\lim_{i\to\infty}q(r_i)=0
\end{equation*}
which contradicts the assumption that $\alpha>0$. Hence case 1 does not hold.

\noindent $\underline{\text{\bf Case 2}}$: $\underset{\substack{i\to\infty}}{\limsup}|I_1(r_i)|=\infty$.

\noindent Then we may assume without loss of generality that $\lim_{i\to\infty}|I_1(r_i)|=\infty$.
Since $\alpha>0$, by \eqref{h-origin-blow-up-rate}, \eqref{qri-sequence-limit}, \eqref{q-integral-represenation} and the l'Hospital rule,
\begin{align*}
&-\alpha=\lim_{i\to\infty}q(r_i)=\lim_{i\to\infty}\frac{-\frac{F(r_i)}{2r_i^2}\left(q(r_i)^2-\frac{(n-1)(h(r_i)-1)}{h(r_i)}\right)}{-r_i^{-2}F(r_i)+r_i^{-1}F(r_i)\left(\frac{\lambda}{2 h(r_i)}+\frac{n-1}{2r_i h(r_i)} \right)}=\frac{\alpha^2-(n-1)}{2}\notag\\
\Rightarrow\quad&\alpha^2+2\alpha-(n-1)=0\notag\\
\Rightarrow\quad&\alpha=\sqrt{n}-1
\end{align*}
 and the theorem follows.

{\hfill$\square$\vspace{6pt}}

\section{Asymptotic behaviour of the function $a(t)$ near the origin}
\setcounter{equation}{0}
\setcounter{thm}{0}

In this section  we will prove the asymptotic behaviour of $a(t)$ near the origin. 

\begin{prop}\label{hr-neg-near-origin-prop1}
Let $2\le n\in\Z^+$, $\alpha=\sqrt{n}-1$, $\lambda\ge 0, c_1\in\Rr$, $c_0>0$  and let $c_2$ be given by \eqref{c2-defn}. For $n>4$, let $h\in C^2((0,\infty))$  be the unique solution of \eqref{h-eqn} in $(0,\infty)$ which satisfies \eqref{h-origin-blow-up-rate} and 
\eqref{h-asymptotic-near-origin1} for some constant  $0<\delta_0<1$ given by 
Theorem \ref{h-asymptotic-behaviour-thm1}. For $n\in\{2,3,4\}$ let $h\in C^2((0,\infty))$ be given by Theorem \ref{h-existence-thm} which satisfies \eqref{h-asymptotic-near-origin2} for some constant $0<\delta_0<1$. Then
\begin{equation}\label{a-near-origin-behaviour}
a(t)\approx (\sqrt{nc_0}t)^{1/\sqrt{n}}\quad\mbox{as }t\to 0^+.
\end{equation}
\end{prop}
\begin{proof}
By \eqref{h-asymptotic-near-origin1}  and \eqref{h-asymptotic-near-origin2},
\begin{align}\label{h-inverse-expansion}
(h(\rho^2))^{-1/2}=&\left\{\begin{aligned}
&c_0^{-1/2}\rho^{\alpha}\left(1+O(\rho^{2\alpha})\right)^{-1/2}\qquad\quad\,\,\mbox{ if }n\ne 4\\
&c_0^{-1/2}\rho\left(1+O(\rho^4|\log\rho|^2)\right)^{-1/2}\quad\mbox{ if }n=4
\end{aligned}\right.\notag\\
=&\left\{\begin{aligned}
&c_0^{-1/2}\left(\rho^{\alpha}+O(\rho^{3\alpha})\right)\qquad\qquad\quad\,\,\mbox{ if }n\ne 4\\
&c_0^{-1/2}\left(\rho+O(\rho^5|\log\rho|^2)\right)\quad\qquad\mbox{ if }n=4.
\end{aligned}\right.
\end{align} 
By \eqref{t-at-relation} and \eqref{h-inverse-expansion},
\begin{equation*}
t\approx\frac{a(t)^{\sqrt{n}}}{\sqrt{nc_0}}\quad\mbox{ as }t\to 0^+
\end{equation*}
and \eqref{a-near-origin-behaviour} follows.

\end{proof}

By a similar argument we have the following proposition.

\begin{prop}\label{hr-neg-near-origin-prop2}
Let $2\le n\in\Z^+$, $\alpha=\sqrt{n}-1$, $\lambda, c_1\in\Rr$, $c_0>0$  and let $c_2$ be given by \eqref{c2-defn}. For $n>4$, let $h\in C^2((0,\ve])$  be the unique solution of \eqref{h-eqn} in $(0,\ve]$ which satisfies \eqref{h-origin-blow-up-rate} and 
\eqref{h-asymptotic-near-origin1} for some constants  $0<\delta_0<\ve<1$, given by 
Proposition \ref{h-asymptotic-behaviour-near-0-prop1}. For $n\in\{2,3,4\}$ let $h\in C^2((0,\ve])$ be given by Corollary \ref{h-local-existence-cor1} which satisfies \eqref{h-asymptotic-near-origin2} for some constants $0<\delta_0<\ve<1$. Then \eqref{a-near-origin-behaviour} holds.
\end{prop}

\end{document}